\documentclass[12pt,a4paper]{amsart}

\usepackage{tikz}
\usetikzlibrary{arrows}
\usetikzlibrary{calc}

\usepackage{mathrsfs}
\usepackage{enumerate}

\usepackage{graphicx}
\usepackage{amsmath}
\usepackage{verbatim}
\usepackage{colonequals}

\usepackage{color}
\long\def\colred#1\endred{{\color{red}#1}}
\long\def\colgreen#1\endgreen{{\color{green}#1}}
\long\def\colblue#1\endblue{{\color{blue}#1}}
\long\def\colmagenta#1\endmagenta{{\color{magenta}#1}}

\theoremstyle{plain}
\newtheorem{thm}{Theorem}[section]

\newtheorem{lemma}[thm]{Lemma}
\newtheorem{corollary}[thm]{Corollary}
\newtheorem{prop}[thm]{Proposition}
\newtoks\prt

\theoremstyle{definition}

\newtheorem{remark}[thm]{Remark}

\newtheorem{definition}[thm]{Definition}

\def\eqn#1$$#2$${\begin{equation}\label#1#2\end{equation}}

\numberwithin{equation}{section}

\let\ccolon=\colon
\let\ccolonequals=\colonequals
\def\adj{\operatorname{adj}}

\def\Adj{\operatorname{Adj}}
\def\ADJ{\operatorname{ADJ}}

\def\BVd{\dot{BV}}

\def\dadj{\mathcal{AD\!J}}

\def\djac{\mathcal{J}}
\def\Det{\operatorname{Det}}
\def\J{\mathcal{J}}
\def\D{\mathcal D}

\def\diam{\operatorname{diam}}
\def\dist{\operatorname{dist}}
\def\dive{\operatorname{div}}
\def\e{{\bf e}}
\def\epsilon{\varepsilon}
\def\ep{\varepsilon}
\def\en{\mathbb N}
\def\qe{\mathbb Q}
\def\er{\mathbb R}

\def\ff{\varphi}

\def\haus{\mathcal{H}}

\def\loc{\operatorname{loc}}
\def\M{\mathcal{M}}
\def\mir2{\mathcal{L}^2}

\def\phi{\varphi}

\def\r2{\er^2}
\def\sgn{\operatorname{sgn}}

\def\spt{\operatorname{supp}}
\def\supp{\operatorname{supp}}

\def\rn{\mathbb R^n}
\def\spt{\operatorname{supp}}
\def\sgn{\operatorname{sgn}}

\def\zet{\mathbb Z}

\newcommand{\inv}{^{-1}}

\newtoks\by
\newtoks\paper
\newtoks\book
\newtoks\jour
\newtoks\yr
\newtoks\pages
\newtoks\vol
\newtoks\publ
\def\ota{{\hbox\vol{???}}}
\def\cLear{\by=\ota\paper=\ota\book=\ota\jour=\ota\yr=\ota
  \pages=\ota\vol=\ota\publ=\ota}
\def\endpaper{\the\by, {\the\paper},
  \textit{\the\jour} \textbf{\the\vol} (\the\yr), \the\pages.\cLear}
\def\endbook{\the\by, \textit{\the\book}, \the\publ.\cLear}
\def\endprep{\the\by, \textit{\the\paper}, \the\jour.\cLear}
\def\endyearprep{\the\by, \textit{\the\paper}, \the\jour, (\the\yr).\cLear}
\def\name#1#2{#2 #1}

\title
{On distributional adjugate and derivative of the inverse}
\author[S. Hencl]{Stanislav Hencl}
\address{Department of Mathematical Analysis, Charles University, Sokolovsk\'{a} 
  83, 186 00
  Prague 8, Czech Republic}
\email{hencl@karlin.mff.cuni.cz}

\author[A. Kauranen]{Aapo Kauranen}
\address{Department of Mathematical Analysis, Sokolovska 83, Praha 8, 186 75, Charles University in Prague \and 
  Departament de Matem\`atiques, Universitat Aut\`onoma de Barcelona, 08193, Bellaterra (Barcelona), Spain}
\email{aapo.p.kauranen@gmail.com}

\author[J. Mal\'y]{Jan Mal\'y}
\address{Department of Mathematical Analysis, Charles University, Sokolovsk\'{a} 
  83, 186 00 Prague 8, Czech Republic}
\email{maly@karlin.mff.cuni.cz}

\thanks{SH and JM were supported by the grant GA\v{C}R P201/18-07996S. 
AK acknowledges financial support from the Spanish Ministry of
Economy and Competitiveness, through the ``Mar\'{\i}a de Maeztu'' Programme for Units of
Excellence in R\&D (MDM-2014- 0445).}

\begin{document}

\begin{abstract}
Let $\Omega\subset\er^3$ be a domain and let $f\colon\Omega\to\er^3$ be a 
bi-$BV$ homeomorphism.
Very recently in \cite{HKL} it was shown that the distributional 
adjugate of $Df$ (and thus also of $Df^{-1}$) is a matrix-valued measure.
In the present paper we show that the components of 
$\Adj Df$
are equal to components of $Df^{-1}(f(U))$ as measures and that the absolutely continuous part of the distributional adjugate 
$\Adj Df$ 
equals to the pointwise adjugate $\adj Df(x)$ a.e.
We also show the equivalence of several approaches to the definition of the distributional adjugate.
\end{abstract}

\maketitle

\section{Introduction}
Suppose that $\Omega\subset \rn$ is an open set
and let $f \ccolon \Omega\to f(\Omega)\subset \rn$ be a homeomorphism. 
In this paper we study the weak differentiability of the inverse of a 
Sobolev or $BV$-homeomorphism. This problem is of 
particular importance as Sobolev and $BV$ spaces are commonly used as initial
spaces for existence problems in PDE's and the calculus of variations.
For instance, elasticity is a typical field where
both invertibility problems and Sobolev (or $BV$) regularity issues 
are relevant (see e.g.\ \cite{Ball}, \cite{CiaNec} and \cite{MST}).

The problem of the weak regularity of the inverse has attracted a big attention 
in the past decade. It started with the result of \cite{HK}, \cite{HKO} and \cite{CHM}  where 
it was shown that for homeomorphisms we have
\eqn{sobolev}
$$
\begin{aligned}
  &\bigl(f\in BV_{\loc}(\Omega,\er^2)\Rightarrow f^{-1}\in 
  BV_{\loc}(f(\Omega),\er^2)\bigr)\\
  &\text{ and }\bigl(f\in W^{n-1}_{\loc}(\Omega,\rn)\Rightarrow f^{-1}\in 
  BV_{\loc}(f(\Omega),\rn)\bigr).\\
\end{aligned}
$$
Moreover, it was shown there that these results are sharp in the scale of Sobolev spaces and moreover under additional assumption one can prove that even $f^{-1}\in W^{1,1}$.

By results of 
\cite{DDSS},
\cite{DS} and \cite{FMS} we know that for $f\in W^{1,n-1}$ we have not only $f^{-1}\in BV$ 
but also the total variation of the inverse satisfies
\eqn{adjformula}
$$
|Df^{-1}|(f(\Omega))=\int_{\Omega}|\adj Df(x)|\; dx
$$
the where $\adj A$ denotes the adjugate matrix to $A$, i.e. the matrix of $(n-1)\times(n-1)$ 
subdeterminants arranged in such a way that
$$
A\adj A=I\det A.
$$
This indicates that the adjugate of $Df$ could be significant also 
for the problem of existence of $Df^{-1}$. One could think
that the integrability of $|D f|^{n-1}$ in \eqref{sobolev}
is needed only to guarantee the integrability of $\adj Df$.
However, for $n\geq 3$ it is possible to construct a $W^{1,1}$ 
homeomorphism with $\adj Df\in L^1$ such that $f^{-1}\notin BV$ (see
\cite{H2}). 
The existence of such a mapping motivates a distributional  approach to the adjugate of the gradient matrix in the problem of characterization of the $BV$-regularity of the inverse.

We use the symbols $\Adj Df$, $\ADJ Df$ and $\dadj Df$ for various
versions of this concept. The definitions of the distributional adjugate
are presented and compared in Section \ref{sec:defs}. In fact,
we show that they are equivalent within the class of measures.

The distributional approach has been successfully used in \cite{HKL} to find
a necessary 
and sufficient 
condition for the $BV$ regularity of 
the inverse for $3$-dimensional $BV$ homeomorphisms. 
We conjecture that the assumption of finite Lebesgue area condition is superfluous.

\begin{thm}\label{regularity}
   Let $\Omega\subset\er^3$ be a domain and 
  $f\in \BVd(\Omega,\er^3)$ be a 
  homeomorphism.
Then $f^{-1}\in \BVd(f(\Omega),\er^3)$ if and only if $\dadj Df\in \M(\Omega,\er^{3\times 3})$ and 
   $f$ satisfies the finite Lebesgue area condition of Section \ref{sec:LebArea}.
\end{thm}


Here $\M$ stands
for the class of all finite (possibly signed or vector-valued) 
Radon measures 
and $\BVd(\Omega,\er^3)$ is the homogeneous $BV$ space, namely, the class
of all $BV_{\loc}$ mappings $f$ on $\Omega$ such that the total variation
of $Df$ is finite, whereas the global integrability of $f$ is not required.
Note that the integrability of $f$ is an issue only if $\Omega$ or
$f(\Omega)$ is unbounded.



The classical inverse mapping theorem states that the formula
\eqn{classical}
$$
\nabla f^{-1}(f(x))J_f(x)=\adj\nabla f(x)
$$
holds for $f$ if $f$ is a regular $C^1$ mapping. 
We are interested in validity of the corresponding formula in the $BV$ setting.
Since the objects considered in \eqref{classical} keep sense only as measures,
the formula should be rewritten as
\eqn{main}
$$
 Df^{-1}(f(U))=\Adj Df(U) \text{ for all open sets }U\subset\Omega.
$$ 
In the planar case, the formula \eqref{main} 
has been proved by Quittnerov\'a \cite{Quit}. 
Even stronger 
results are established by
D'Onofrio, Mal\'y, Sbordone and Schiattarella \cite{DMSS}. 
For $n>2$, the formula \eqref{main}
has been obtained by Quittnerov\'a \cite{Quit}
under the assumption that $f\in W^{1,n-1}$.
Our main goal is to prove \eqref{main} in $\er^3$ assuming only that 
$f$ and $f^{-1}$ are $BV$.

\begin{thm}\label{t:main}
Let $\Omega\subset\er^3$ be a domain and 
  $f\in \BVd(\Omega,\er^3)$ be a 
  homeomorphism such that $f^{-1}\in \BVd(f(\Omega),\er^3)$.
    Then  
\eqn{main1}    
$$
 Df^{-1}(f(U))=\Adj Df(U) \text{ for all open sets }U\subset\Omega.
$$ 
\end{thm}

In coordinates, \eqref{main1} reads as 
\eqn{mainij}
$$
(D_j(f^{-1})_i)(f(U))
=(\Adj_{ij} Df) (U),\qquad i,j\in\{1,2,3\}.
$$
This improves the result of \cite{DS} where the equality of variations
in \eqref{mainij} is shown 
for $W^{1,n-1}$ homeomorphism for $n\geq 3$.

As a corollary of our results we show that in some cases 
it is possible to verify the somewhat technical assumption $\Adj Df\in \M$ easily using coordinate functions  $f=(f_1,f_2,f_3)$. 
Moreover, under the same assumptions the finite Lebesgue area condition holds and we have the $BV$ regularity of the inverse.

\begin{thm}\label{cor}
  Let $\Omega\subset\er^3$ be a domain and $f\in \BVd(\Omega,\er^3)$ be a 
  continuous mapping.
  Assume 
  that 
\par  
{\rm(a)}  $f_i\in W^{1,p_{i}}(\Omega)$ 
  where $p_1,p_2,p_3\in [1,\infty]$ and $\frac1{p_i}+\frac1{p_j}\le 1$
 for each distinct $i,j$ (with the convention $\frac1\infty=0$),\par\noindent
 or that \par
{\rm(b)} 
at least two coordinates of $f$ are in $C^1(\overline\Omega)$.\par\noindent
Then $\Adj Df\in \M(\Omega,\er^{3\times 3})$ and $f$ satisfies the finite Lebesgue area condition (see Sec. \ref{sec:LebArea}).
Therefore, $f^{-1}\in \BVd(f(\Omega),\er^3)$ if $f$ is a homeomorphism.
\end{thm}

It is known that the absolutely continuous part of the distributional Jacobian equals to the pointwise Jacobian a.e.\ for nice enough $f$
(see De Lellis \cite[Lemma 4.7]{D} and M\"uller \cite{M}). 
Similar statement holds also for the distributional adjugate. 

\begin{thm}\label{acpart}
  Let $\Omega\subset\er^3$ be a domain and 
  $f\in \BVd(\Omega,\er^3)$ be a 
  continuous mapping such that $\Adj Df\in\M(\Omega,\er^{3\times 3})$.
   Then the absolutely continuous part of  $\Adj Df$ 
(with respect to Lebesgue measure) equals to the pointwise 
adjugate  $\adj Df(x)$ for a.e. $x\in\Omega$. 
\end{thm}

The proofs of Theorems \ref{t:main}, \ref{cor} and \ref{acpart} are given in Section \ref{s:cor}.
Since the proof of \eqref{mainij} is the same for all choices of coordinates,
we demonstrate it on the choice $i=j=3$, which leads to the most comfortable notation.
The method is to express both parts of \eqref{mainij} in terms of degree.
The left hand part is handled in Section \ref{s:grad}, the right hand part in Section
\ref{s:Adj}. We need a two-dimensional degree formula for the distributional Jacobian
derived in Section \ref{s:dim2}. The comparison of various definitions of the distributional
adjugate is given in Section \ref{sec:defs}. To complete the list, Section \ref{s:pre}
is devoted to various preliminaries.

\section{Preliminaries}\label{s:pre}

For a domain $\Omega \subset \rn$ we denote by  $\D(\Omega)$ those smooth
functions $\varphi$ whose support is compactly contained in $\Omega$,
i.e.\ $\supp \varphi \subset \subset \Omega$.

Given a distribution $T$ on an open set $\Omega$, 
the action of $T$ on a test function
$\ff\in\D(\Omega)$ is denoted by $\langle T,\ff\rangle$. This can be extended to 
more general test functions according to the quality of $T$, for example,
to $T$-integrable test functions if $T$ is a measure.

The total variation of an $\rn$-valued Radon measure $\mu$ is the measure $|\mu|$ such that
$$
\langle |\mu|,\psi\rangle\ccolonequals  \sup\left\{\int_{\rn} \phi \cdot d\mu:\ \phi\in C_0(A;\rn),\ 
|\phi|\le \psi\right\},\qquad \psi\in C_0^+(\rn).
$$
Given two vectors $u,v\in\er^3$ we denote by $u\times v$ their cross product, defined by the property
$$
w\cdot (u\times v)=\det(w,u,v),\qquad w\in\er^3.
$$

\subsection{Slicing of BV function}\label{ss:slicing}
Let $f\ccolon\Omega\to\er^m$ be a BV function and $\ff\in\D(\Omega)$.
For simplicity we assume that $\Omega=(0,1)^3$.
Then
$$
\langle 
D_1f,\;\ff
\rangle
=
\int_{(0,1)^2}\langle D_1f(\cdot,x_2,x_3),\;\ff(\cdot,x_2,x_3)\rangle\,dx_2\,dx_3
$$
and
$$
\langle 
|D_1f|,\;\ff
\rangle
=
\int_{(0,1)^2}\langle |D_1|f(\cdot,x_2,x_3),\;\ff(\cdot,x_2,x_3)\rangle\,dx_2\,dx_3,
$$
see e.g.\ \cite[Theorem 3.103]{AFP}.
Integrating with respect to $x_2$ we obtain
$$
\langle 
D_1f,\;\ff
\rangle
=
\int_{(0,1)}\langle D_1f(\cdot,\cdot,x_3),\;\ff(\cdot,\cdot,x_3)\rangle\,dx_3.
$$
and
$$
\langle 
|D_1f|,\;\ff
\rangle
=
\int_{(0,1)}\langle |D_1f|(\cdot,\cdot,x_3),\;\ff(\cdot,\cdot,x_3)\rangle\,dx_3.
$$
Similarly we express $D_2f$ by integration over $x_3$
(but not $D_3f$). By approximation we observe that 
these identities can be extended to test functions $\ff\in C_0(\Omega)$.

\subsection{Topological degree}
\label{sec:TopologicalDegree}\
For 
a bounded open set
$\Omega \subset \rn$ and a given smooth map $f \ccolon 
\overline\Omega 
\to \rn$
we define the \textit{topological degree} as
\begin{align*}
  \deg(f,\Omega, y_{0}) = \sum_{x \in \Omega \cap f\inv\{y_0\}} \sgn(J_{f}(x))
\end{align*}
for a point $y_0\in\rn\setminus f(\partial\Omega)$
if $J_{f}(x) \neq 0$ for each $x \in f^{-1}(y_{0})$. This definition can be extended to arbitrary continuous mappings and each point $y_0\notin f(\partial\Omega)$, see e.g.\ \cite[Section 1.2]{FG} or \cite[Chapter 3.2]{HKbook}. For our purposes the following property of the topological degree is crucial;
see \cite[Definition 1.18]{FG}.
\begin{lemma}
  \label{lemma:TopologicalDegree}\ 
  Let $\Omega \subset \rn$ be a 
  bounded open set and
  $f \ccolon\overline\Omega\to \rn$ 
  be a continuous function.
  Then for any point $y_0 \in \rn \setminus 
  f (\partial \Omega)$
  and any continuous mapping $g \ccolon \overline\Omega\to \rn$ satisfying
  \begin{align*}
    | f - g |
    \leq \operatorname{dist}\left( y_0 , 
    f(\partial \Omega)
     \right)\quad \text{on }\partial \Omega
  \end{align*}
  we have $\deg(f, \Omega, y_0) = \deg(g, \Omega, y_0)$.
\end{lemma}

Moreover, we need to use also degree composition formula see \cite[Proposition IV.6.1]{OR}.

\begin{lemma}
  \label{composition}
  Let $\Omega \subset \rn$ be a bounded open set.
Let  $h \ccolon \overline\Omega \to \rn$ and $g:\rn\to \rn$ be 
\colblue
a 
\endblue
continuous function. 
  Assume that $y\notin g(h(\partial \Omega))$ and let $\Delta_i$ be the bounded
  connected components of $\rn\setminus h(\partial \Omega)$.
  Then
$$
\deg(g\circ h, \Omega, y) = \sum_i \deg(g, \Delta_i, y) \deg(h,\Omega,\Delta_i).
$$
\end{lemma}

\subsection{Hausdorff measure}
Given $k\ge 0$
we define 
$$
\haus^k(A)=\lim_{\delta\to 0+}\haus^k_{\delta}(A).
$$
where
$$
\haus^k_{\delta}(A)=\inf\Bigl\{\alpha_k\sum_i \Bigl(\tfrac12\diam A_i\Bigr)^k:\ A\subset\bigcup_i A_i,\ \diam A_i\leq \delta\Bigr\}, \qquad 0<\delta\le \infty
$$ 
and 
$$
\alpha_k=\frac{\pi^{k/2}}{\Gamma(1+\frac k2)},
$$
See e.g.\ \cite{Fe}.

\subsection{Degree formula}\label{changevar}
Let $h:\Omega\to\rn$ be a $C^1$ smooth mapping. Then the 
change of variables formula 
\eqn{changeofvariables}
$$
\int_{G} v(h(x))J_h(x)\; dx=\int_{\rn}v(y)\deg(h,G,y)\; dy\ 
$$
holds for each open set 
$G\subset\subset  \Omega$ 
and each measurable $v\ccolon h(\Omega)\to[0,\infty)$.

\subsection{Disintegration}

Let $Q=(0,1)^n$, $\mu\in \M(Q)$ and 
$\nu$ be a nonnegative finite Radon measure on $(0,1)$.
We denote the $k$-dimensional Lebesgue measure by $\lambda_k$. 
We still abbreviate ``$\lambda_k$-a.e.'' as ``a.e.''.
For Lebesgue decomposition of measures we refer to 
\cite[Theorem 1.28]{AFP}.

A system $(\mu_t)_{t\in (0,1)}$, where $\mu_t$ are signed Radon measures on $(0,1)^{n-1}$,
is called a \textit{disintegration} of $\mu$ with respect to $\nu$
if 
 \begin{equation}
  \label{disint}
  \mu(A)=\int_0^1\mu_t(A)\,d\nu(t),
  \qquad A\subset Q \text { Borel.}
 \end{equation}
 Note that this is equivalent to the validity of
 \begin{equation}
  \label{disint1}
  \int_{Q}\ff(y,t)\,d\mu(x,t)=\int_0^1\Bigl(\int_{(0,1)^{n-1}}\ff(y,t)\,d\mu_t(y)\Bigr)\,d\nu(t)
 \end{equation}
for each bounded Borel measurable $\ff\colon Q\to\er$.

\begin{thm}\label{t:disint}
Let $\mu\in \M(Q)$.
Then there exists a disintegration $(\mu_t)_{t\in (0,1)}$ of $\mu$ with respect to
\eqn{hownu}
$$
\nu\colon E\mapsto |\mu|((0,1)^{n-1}\times E), \qquad E\subset (0,1) \text{ Borel}.
$$
Moreover, if $(\mu_t)_t$ and $(\sigma_t)_t$ are disintegrations of $\mu$ with respect to 
$\nu$, then $\mu_t=\sigma_t$ for $\nu$-a.e. $t\in (0,1)$.
\end{thm}

\begin{proof} See e.g.\ \cite[Theorem  2.28]{AFP}.
\end{proof}

\begin{corollary}\label{c:disint}
Let $\mu\in\M(Q)$. Let  $(\mu_t)_t$ and $(\sigma_t)_t$ are disintegrations of $\mu$ with respect to 
the Lebesgue measure $\lambda_1$ on $(0,1)$. Then $\mu_t=\sigma_t$ for a.e. $t\in (0,1)$.
\end{corollary}

\begin{proof} Let $\nu$ be as in \eqref{hownu},
$\rho$ be the absolutely continuous part of $\lambda_1$ with respect to
$\nu$ and $a$ be the Radon-Nikodym derivative of $\rho$ with respect to $\nu$.
Then there is a Borel set $E\subset (0,1)$ such that $\nu(E)=0$
and $\rho=\lambda_1$ on $(0,1)\setminus E$. Then $(a(t)\mu_t)_t$ and $(a(t)\sigma_t)_t$
are disintegrations of $\mu$ with respect to $\nu$. By the uniqueness part of 
Theorem \ref{t:disint} we have $\mu_t=\sigma_t$ $\nu$-a.e.\ in $(0,1)$,
and by the absolute continuity, $\mu_t=\sigma_t$ $\rho$-a.e.\ in $(0,1)$,
which means  $\mu_t=\sigma_t$ a.e.\ in $(0,1)\setminus E$.

For each cube $M\subset (0,1)^{n-1}$ and Borel set $E'\subset E$ we have
$$
\Bigl|\int_{E'}\mu_t(M)\,dt\Bigr|=|\mu(M\times E')|\le \nu(E')=0.
$$
It follows that $\mu_t(M)=0$ for a.e.\ $t\in E$. 
The same argument shows that 
$\sigma_t(M)=0$ for a.e.\ $t\in E$. We find a joint set
$Z\subset E$ of $\lambda_1$-measure $0$ such that $\mu_t(M)=0=\sigma_t(M)$ for $t\in E\setminus Z$
and each cube $M$ from a dense family of cubes in $[0,1]^{n-1}$. It follows that
$\mu_t=\sigma_t$ a.e.\ also in $E$.
\end{proof}

\begin{remark}\label{r:disint}
Let $\mu\in\M(Q)$, $\nu$ be as in \eqref{hownu},
$(\mu_t)_t$ be a disintegration of $\mu$ with respect to $\nu$
and $(|\mu|_t)_t$ be a disintegration of $|\mu|$ with respect to $\nu$.
Then $|\mu|_t=|\mu_t|$ for $\nu$-a.e.\ $t\in (0,1)$.
Indeed, consider
$$
\sigma_t(M)=\int_{M}\theta(y,t)\,d|\mu|_t,\qquad M\subset(0,1)^{n-1} \text{ Borel,}
$$
where $\theta=\frac{d\mu}{|d\mu|}$. Then the claim follows from the uniqueness part of 
Theorem \ref{t:disint}.
Similar observation holds
for the positive and negative parts of $\mu$. 

If follows that $|\mu_t|((0,1)^{n-1})=1$ for a.e.\ $t\in (0,1)$.
\end{remark}

\begin{lemma}\label{l:absdisint} 
Let $Q=(0,1)^n$ and $\mu\in \M(Q)$.
Let $(\mu_t)_t$ be a disintegration of $\mu$ with respect to $\lambda_1$.
Let $\mu_a$ be the absolutely continuous part of $\mu$ with respect to $\lambda_n$ and 
$(\mu_t)_a$ denote the absolutely continuous parts of $\mu_t$, $t\in (0,1)$,
with respect to $\lambda_{n-1}$.
Then $((\mu_t)_a)_t$ is a disintegration of $\mu_a$ with respect to $\lambda_1$.
\end{lemma}

\begin{proof} 
Let $\mu_s$ be the singular part of $\mu$ and $g$
be a Borel-measurable representative of 
the Radon-Nikodym derivative of $\mu_a$ with respect to $\lambda_n$.
Then there is a Borel set $E\subset Q$ of measure zero such that 
\eqn{howE}
$$
\mu_s(A)=\mu (E\cap A),\qquad A\subset Q\text{ Borel.}
$$
By the Fubini theorem, the set 
$$
E_t = \{y\in (0,1)^{n-1}\colon (y,t)\in E\}
$$ 
has $(n-1)$-dimensional measure zero for almost every $t\in (0,1)$. 
Set 
$$
\tilde E=\{(y,t)\in E\colon \lambda_{n-1}(E_t)=0\}.
$$
Then $\tilde E$ can be used in place of $E$ in \eqref{howE}.
Set 
$$
\sigma_t=(\sigma_t)_a+(\sigma_t)_s,
$$
where for each Borel set $M\subset(0,1)^{n-1}$ we define
$$
\aligned
(\sigma_t)_s(M)&=\mu_t(M\cap \tilde E_t),\\
(\sigma_t)_a(M)&=\int_{M}g(y,t)\,dy.
\endaligned
$$
Then for each $t\in (0,1)$, $(\sigma_t)_a$ is absolutely continuous with respect to 
$\lambda_{n-1}$ and $(\sigma_t)_s$ is singular with respect to 
$\lambda_{n-1}$. It is easily seen that $(\sigma_t)_t$ is a disintegration of $\mu$ 
with respect to $\lambda_1$ and thus 
by Corollary \ref{c:disint}, $\sigma_t=\mu_t$ for a.e. $t\in (0,1)$. 
It follows that  $((\mu_t)_a)_t$ is a disintegration of $\mu_a$ with respect to 
the Lebesgue measure on $(0,1)$.

\end{proof}


\subsection{Lebesgue area}
\label{sec:LebArea}
Let $\Omega\subset \er^2$.
If $g\ccolon\Omega\rightarrow \er^3$ is a piecewise linear continuous map, 
we define the \emph{Lebesgue area} of $g$ by 
$$
L(g)=\sum_{T\in \Delta} \haus^2(f(T)), 
$$
where $\Delta$ is any triangulation of $\Omega$ for which  $g$ is linear in every triangle $T\in \Delta.$
For a general continuous map $g\ccolon \Omega\rightarrow \er^3$ we set 
$$
L(g)=\lim_{\epsilon\rightarrow 0} \inf\{L(h)\ccolon\,  
h \text{ piecewise linear, } \|h-g\|_\infty<\epsilon\}.
$$

Assume now that $\Omega\subset \er^3$ and $f:\Omega\rightarrow \er^3$ is a continuous mapping. We say that $f$ satisfies the \emph{finite Lebesgue area condition} if for almost every $t\in \er$ the mappings $f(t,\cdot,\cdot),\,f(\cdot,t,\cdot)$ and $f(\cdot,\cdot,t)$ have finite Lebesgue area.
\section{Distributional Jacobian}\label{s:dim2}

Let $G\subset\er^2$ be open and $g\in \BVd(G,\er^2)$ be continuous.
Then $\adj Dg$ is the matrix-valued measure
$$
\adj Dg = 
\begin{pmatrix}
D_2g_2,&-D_2g_1\\
-D_1g_2,&D_1g_1
\end{pmatrix}.
$$
The distributional Jacobian of 
$g$ 
is the limit 
$$
\Det Dg=\lim_{k\to\infty}\det(\nabla g_k)
$$
in distributions,
where $g_k\to g$ are standard mollifications of $g$. 
We use also the symbol $\djac_g$ for $\Det Dg$.
A routine approximation gives
\eqn{distrjac}
$$
\langle\djac_g,\varphi\rangle=\int_G \varphi(x) J_g(x)\; dx
$$ 
if $g$ is smooth enough, e.g.\ $g\in W^{1,2}(G)$. Under standing assumptions,
\eqref{distrjac} can fail but
we can integrate by parts to obtain
\eqn{genadj}
$$
\langle \djac_g,\ff\rangle = -\sum_{i,j=1}^2
\langle\adj_{ij} Dg,\;(\Phi_j\circ g)\, D_i\ff \rangle
,\qquad \ff\in \D(G),
$$
if $\Phi\ccolon\er^2\to\er^2$ is a $C^1$-mapping satisfying $\dive\Phi=1$ on a neighborhood of
$g(\overline\Omega)$.
Indeed, for smooth function we can refer to 
\cite{MTY} to the formula
\eqn{neplati}
$$
\sum_{i,j=1}^2D_i\bigl(\adj_{ij}Dg \;\Phi_j\circ g\bigr)=(\dive\Phi)\circ g\;\det Dg
$$
and passing to the limit in duality between measures and continuous functions we obtain the general case.
(Note that, in our generality, the passage to the limit on the right is not guaranteed unless $\dive\Phi=1$.)
In particular, the choice $\Phi(y)=y_1$ yields
\eqn{defdistr}
$$
\aligned
\langle \djac_g,\ \varphi\rangle 
=\langle D_1g_2,\;g_1D_2\ff\rangle
-\langle D_2g_2,\;g_1D_1\ff\rangle,
\qquad\varphi \in \D(G).
\endaligned
$$

\subsection{Two-dimensional degree and the Distributional Jacobian}\label{ss:distrjac}

\begin{lemma}\label{l:W}
Let $W\subset\er^2$ be a bounded open set and $g\in C(\overline W,\er^2)\cap BV(W,\er^2)$. Let $\eta\in \D(\er^2)$ have support in $\er^2\setminus g(\partial W)$. 
Let 
$\Phi\ccolon\er^2\to\er^2$ be a $C^1$ mapping such that 
$$
\dive \Phi=\eta
$$
and $\ff\in \D(W)$ be such that 
$\overline {\{\ff\ne 1\}}\cap\overline{ \{\eta\circ g\ne 0\}}=\emptyset$. 
Then 
$$
-\sum_{i,j=1}^2
\langle\adj_{ij} Dg,\;(\Phi_j\circ g)\, D_i\ff \rangle
=\int_{\er^2}\eta(y)\deg(g,W,y).
$$
\end{lemma}

\begin{proof}
If $g$ is smooth, we have by analogy of \eqref{distrjac}, \eqref{defdistr} and \eqref{changeofvariables}
\eqn{smoothdegree}
$$
\aligned
&
-\sum_{i,j=1}^2
\langle\adj_{ij} Dg,\;(\Phi_j\circ g)\, D_i\ff \rangle
\\&\quad
=-\int_{W}(\Phi_2\circ g)\det(\nabla g_1,\nabla \ff)\,dx
+\int_{W}(\Phi_1\circ g)\det(\nabla g_2,\nabla \ff)\,dx
\\&\quad
=\int_{W}(D_2\Phi_2\circ g)\det(\nabla g_1,\nabla g_2)\ff\,dx
-\int_{W}(D_1\Phi_1\circ g)\det(\nabla g_2,\nabla g_1)\ff\,dx
\\&
\quad
=\int_{W}
\dive \Phi(g(x))\,J_g(x)\,\ff(x)\,dx
=\int_{W}
\eta(g(x))\,J_g(x)\,dx
\\&\quad
=\int_{\er^2}\eta(y)\,\deg(g,W,y)\,dy.
\endaligned
$$
In the general case we approximate $g$ by standard mollifications
$g^{(j)}$.
The passage to the limit on the left of \eqref{smoothdegree}
is easy, as $\Phi\circ g^{(j)}\to
\Phi\circ g$ uniformly and $Dg^{(j)}\to Dg$ weak* in measures.
The passage on the right follows from the fact that $g^{(j)}\to g$
uniformly and $\eta$ has compact support in $\er^2\setminus g(\partial W)$ (see Lemma \ref{lemma:TopologicalDegree}).
\end{proof}

\begin{corollary}\label{c:Nest}
Let $W\subset\er^2$ be a bounded open set and $g\in C(\overline W,\er^2)\cap BV(W,\er^2)$. Let $\eta\in \D(\er^2)$ have support in $\er^2\setminus g(\partial W)$. 
Let 
$\Phi\ccolon\er^2\to\er^2$ be a $C^1$ mapping such that 
$$
\dive \Phi=\eta
$$
and $\ff\in \D(W)$ be such that 
$\overline {\{\ff\ne 1\}}\cap\overline{ \{\eta\circ g\ne 0\}}=\emptyset$. 
Then 
$$
\Bigl|\sum_{i,j=1}^2
\langle\adj_{ij} Dg,\;(\Phi_j\circ g)\, D_i\ff \rangle\Bigr|
\le
||\eta||_\infty
\int_{\er^2}
|\deg(g,W,y)|\,dy.
$$
\end{corollary}

\begin{lemma} 
\label{l:atx1}
Let  $Q=Q(\bar x,r)$ be a square in $\er^2$ and $0<\rho_k<r$,
$\rho_k\nearrow r$.
Let $g\in BV(Q,\er^2)$ be a continuous $BV$ mapping.
Let $\eta,\eta_k\in \D(\er^2)$.
Suppose that $\eta=1$ on a neighborhood of $g(\overline W)$ and 
$\eta_k\in \D(\er^2)$ have support in $\er^2\setminus g(\overline Q\setminus Q(\bar x,\rho_k))$. Let 
$\Phi,\,\Phi^{(k)}\ccolon\er^2\to\er^2$ be  $C^1$ mappings such that 
$$
\dive \Phi^{(k)}=\eta_k,\quad \dive \Phi=\eta
$$
and $\ff_k\in \D(Q)$ be such that $\ff_k=1$ on $Q(\bar x,\rho_k)$. 
Suppose that  
$\J_g\in \M(Q)$  
and
\eqn{lim1}
$$
\Phi_1^{(k)}\to \Phi
\text{ uniformly on }g(\overline Q),
$$
\eqn{lim2}
$$
|\nabla\ff_k|\le \frac{C}{r-\rho_k}
$$
and
\eqn{lim3}
$$
\limsup_{k\to\infty}\frac{|Dg|(Q\setminus Q(\bar x,\rho_k))}{r-\rho_k}<\infty.
$$
Then
$$
\aligned
-\lim_{k\to\infty} &
\Bigl(
\sum_{i,j=1}^2
\langle\adj_{ij} Dg,\;(\Phi_j^{(k)}\circ g)\, D_i\ff \rangle
\Bigr)
=\J_g(Q).
\endaligned
$$
\end{lemma}

\begin{proof}
Taking into account that (see \eqref{genadj})
$$
-\sum_{i,j=1}^2
\langle\adj_{ij} Dg,\;(\Phi_j\circ g)\, D_i\ff\rangle =\langle \J_g,\ff_j\rangle,
$$
we have
$$
\aligned
&\Bigl|-\sum_{i,j=1}^2
\langle\adj_{ij} Dg,\;(\Phi_j^{(k)}\circ g)\, D_i\ff \rangle 
-\J_g(Q)\Bigr|
\\&\quad
\le 
\Bigl|\sum_{i,j=1}^2
\langle\adj_{ij} Dg,\;(\Phi_j^{(k)}\circ g-\Phi_j\circ g)\, D_i\ff \rangle\Bigr|
\\&\quad
+ 
\bigl|\langle \J_g,\ff_j\rangle-\J_g(Q)\bigr|\to 0.
\endaligned
$$
The second term is easy, for the first on we use \eqref{lim1}--\eqref{lim3}.
\end{proof}

\begin{lemma}\label{l:Lp}
Let 
\eqn{howK}
$$
L(x)=-\frac1{2\pi}\log|x|,\quad K(x)=-\frac{1}{2\pi}\,\frac{x}{|x|^2},\qquad x\in \er^2\setminus\{0\}.
$$
Then 
$$
\dive (K*\psi) = \psi,\qquad \psi\in \D(\rn) \text{ supported in }B(0,R)
$$
and
\eqn{inftyest}
$$
|K*\psi(x)|\le C R^{1/2}\|\psi\|_{L^3(B(0,R))},\qquad x\in\er^2.
$$
\end{lemma}

\begin{proof}
Let 
$u=L*\psi$ be the Newtonian (alias logarithmic) potential of 
$\psi$.
Then
$$
\dive K*\psi = -\Delta u=\psi.
$$
The estimate \eqref{inftyest} follows from the H\"older inequality as
$$
\|K\|_{L^{3/2}(B(x,R))}\le CR^{1/2}.
$$
\end{proof}

\begin{thm}\label{t:twodim} 
Let  $Q=Q(\bar x,r)$ be a square in $\er^2$.
Let $g\in BV(Q,\er^2)\cap C(\overline Q,\er^2)$.
Suppose that 
\eqn{twodim}
$$
|g(\partial Q)|=0
$$
and
\eqn{lim3a}
$$
s:=\sup_{0<\rho<r}\frac{|Dg|(Q\setminus Q(\bar x,\rho))}{r-\rho}<\infty.
$$
Then
$$
\int_{\er^2}\deg(g,Q,y)=\J_g(Q).
$$
\end{thm}

\begin{proof}
Let $B(0,R)$ be a ball containing $g(\overline Q)$.
Let $\eta$ be a smooth function with support in $B(0,R)\setminus g(\partial Q)$ 
such that  $|\eta|\le 1$.
Set $\Phi=K*\eta$, where $K$ is as in \eqref{howK}.
Then $\dive\Phi = \eta$ by Lemma \ref{l:Lp}.
We find $\rho< r$ such that 
$$
\overline{\{\eta\circ g\ne 0\}}\subset  Q(\bar x,\rho)
$$
and a test functions $\ff\in\D(Q)$ such that $0\le \ff\le 1$,
$\ff=1$ on $ Q(\bar x,\rho)$ and 
$$
|\nabla\ff|\le \frac{C}{r-\rho}.
$$
By Lemma \ref{l:W} we have
$$
\int_{\er^2}\eta(y)\,\deg(g,Q,y)\,dy=
-\langle\sum_{i,j=1}^2
\langle\adj_{ij} Dg,\;(\Phi_j\circ g)\, D_i\ff \rangle,
$$
and thus from \eqref{lim3a} and Lemma \ref{l:Lp} we infer that
$$
\Bigl|\int_{\er^2}\eta(y)\,\deg(g,Q,y)\,dy\Bigr|
\le C\sup_{y\in\er^2}|\Phi(y)|\;\sup_{0<\rho<r}\frac{|Dg|(Q\setminus Q(\bar x,\rho))}{r-\rho}\le CsR^{1/2}.
$$
Since this holds for all functions $\eta$ with the above listed properties, we deduce that 
\eqn{degest}
$$
\int_{\er^2}|\deg (g,Q,y)|\,dy<\infty.
$$
Now, let $\eta_0$ be a smooth function with compact support such that
$0\le \eta_0\le 1$ and $\eta_0=1$ on a neighborhood of $g(\overline Q)$. 
Consider a 
sequence $\eta_k$ of smooth functions such that $\eta_k=0$ on a neighborhood of
$g(\partial Q)$, $k=1,2,\dots$,
$0\le \eta_1\le \eta_2\le \dots\le \eta_0$ and 
$\eta_k\to\eta_0$ a.e. Let $K$ be as in \eqref{howK}. Set
$$
\Phi^{(k)}=K*\eta_k,\qquad k=1,2,\dots.
$$
From Lemma \ref{l:Lp} we obtain that 
$$
\dive\Phi^{k} = \eta_k\quad \text{in }g(\overline Q)
$$
and that  $\Phi^{(k)}\to \Phi$ uniformly
on $\overline {g(Q)}$ as $\eta_k\to \eta_0$ in $L^3(B(0,R))$.
Next, we find $\rho_k\nearrow r$ such that 
$$
\overline{\{\eta_k\circ g\ne 0\}}\subset  Q(\bar x,\rho_k)
$$
and test functions $\ff_k\in\D(Q)$ such that $0\le \ff_k\le 1$,
$\ff_k=1$ on $ Q(\bar x,\rho_k)$ and 
$$
|\nabla\ff_k|\le \frac{C}{r-\rho_k}.
$$
By Lemma \ref{l:W} we have
$$
- \sum_{i,j=1}^2
\big\langle\adj_{ij} Dg,\;(\Phi_j^{(k)}\circ g)\, D_i\ff\big \rangle=\int_{\er^2}\eta_k(y)\,\deg(g,Q,y)\,dy
$$
and passing to the limit as $k\to\infty$ and obtain 
$$
\J_g(Q)=\int_{\er^2}\,\deg(g,Q,y)\,dy.
$$
Indeed, the passage to the limit on the left follows from Lemma \ref{l:atx1} and  
the passage to the limit on the right is justified by \eqref{degest}. 
\end{proof}

\begin{remark}
Since for continuous $g\in BV(\Omega,\er^2)$,
``almost every'' square $Q\subset\Omega$ satisfies \eqref{lim3a}, we have obtained an
alternative proof of \cite[Theorem 4.1]{HKL}.
\end{remark}

\section{On various definitions of distributional adjugate}
\label{sec:defs}


Throughout this section, we use the symbol $i'$ for the action of 
the cyclic permutation on $i$, namely $1'=2$, $2'=3$, $3'=1$, $i''=(i')'$.
Also, we use the maps
$$
\kappa_1^t(y)=(t,y_1,y_2),\quad
\kappa_2^t(y)=(y_2,t,y_1),\quad
\kappa_3^t(y)=(y_1,y_2,t),\qquad y\in\er^2,\;t\in\er.
$$

The following notion of the distributional adjugate has been
introduced in \cite{HKL}.

\begin{definition}\label{slicedef}
  Let $f=(f_1,f_2,f_3) \ccolon \Omega\to\er^3$ be a continuous 
  $BV$ mapping.
The distributional adjugate of the first kind of $f$ is defined as 
$$
\langle
\ADJ_{ij} Df,\;\ff
\rangle
=
\int_{-\infty}^{\infty}
\langle
\Det (D(f_{j'}\circ\kappa_i^t),\,D(f_{j''}\circ\kappa_i^t)),\;
\ff\circ \kappa_i^t
\rangle\,dt,\quad \ff\in\D(\Omega).
$$  
Here the duality between $\Det (Df\circ \kappa_i^t)$ and 
$\ff\circ\kappa_i^t$ is considered on $(\kappa_i^t)^{-1}(\Omega)=\{x\in\Omega:\ x_i=t\}$.

We use the symbol $\dadj Df$ for $\ADJ Df$ if we know that
the distributional Jacobians
$\Det (D(f_{j'}\circ\kappa_i^t),\,D(f_{j''}\circ\kappa_i^t))$ are 
signed Radon measures for a.e.\ $t$ and all $i$, $j$.
\end{definition}

Following directly the way how we defined the distributional Jacobian
in Subsection \ref{ss:distrjac},
 we consider another 
approach to the distributional adjugate.

\begin{definition}\label{d:Adj}
Let $f=(f_1,f_2,f_3)\ccolon \Omega\to\er^3$ 
be a continuous $BV$ mapping.
The distributional adjugate of the second kind of $f$ is defined as 
$$
\Adj Df=\lim_{k}\adj \nabla f_k,
$$
where $f_k\to f$ are standard mollifications of $f$ and the convergence is in distributions.
\end{definition}

We can integrate by parts similarly to \eqref{genadj}, in particular we have
\begin{equation}\label{byparts}
 \begin{aligned}
  \langle
  \Adj_{ij} Df,\;\ff
  \rangle
  = 
    \langle
  D_{i'}f_{j''},\; f_{j'} D_{i''}\ff
  \rangle
  -
  \langle
  D_{i''}f_{j''},\; f_{j'} D_{i'}\ff
  \rangle,\qquad \ff\in\D(\Omega).
 \end{aligned}
\end{equation}

\begin{prop}\label{aapo}
 Let $f\in BV(\Omega,\er^3)$ be a continuous mapping, $i,j\in\{1,2,3\}$
 Then 
 \eqn{equality}
 $$
 \ADJ_{ij} Df=\Adj_{ij}Df.
 $$
If $\Adj_{ij} Df\in \M(\Omega)$, then for almost every $t\in \er$ it holds that
the distribution
$\delta_t:=\Det  \bigl(D(f_{j'}\circ\kappa_i^t),\,D(f_{j{''}}\circ\kappa_i^t)\bigr)$ 
is a signed Radon measure 
on $\Omega_t:=(\kappa_i^t)^{-1}(\Omega)$
and 
the function
$$
t\mapsto |\delta_t|(\Omega_t)
$$
is Lebesgue integrable.\newline
Therefore, $\Adj Df=\ADJ Df=\dadj Df$ if 
$\Adj Df\in\M(\Omega)$.
\end{prop}

\begin{proof}
We prove the result only for $i=j=3$ as all the other cases are similar. Without loss of generality we will also assume that $\Omega=(0,1)^3$.
Let $\varphi\in C^\infty_0(\Omega).$ 
Using this $\varphi$ as a test function, for almost every $t\in(0,1)$ we obtain 
$$
  \begin{aligned}
&
\langle
\Det (D(f_{1}\circ\kappa_3^t),\,D(f_{2}\circ\kappa_3^t)),\;
\ff\circ \kappa_3^t
\rangle
\\&\quad
=
\langle D_1f_2(\cdot,\cdot,t),\;f_1(\cdot,\cdot,t)D_2\ff(\cdot,\cdot,t)\rangle
-\langle D_2f_2(\cdot,\cdot,t),\;f_1(\cdot,\cdot,t)D_1\ff(\cdot,\cdot,t)\rangle.
  \end{aligned}
  $$ 
Integrating with respect to $t$ like in Subsection \ref{ss:slicing} 
and using \eqref{byparts}
we obtain
\eqn{integratedadj}
$$
\aligned
\langle
\ADJ_{33} Df,\;\ff
\rangle
&=
\int_0^1
\Bigl(\langle D_1f_2(\cdot,\cdot,t),\;f_1(\cdot,\cdot,t)D_2\ff(\cdot,\cdot,t)\rangle
\\&\qquad-\langle D_2f_2(\cdot,\cdot,t),\;f_1(\cdot,\cdot,t)D_1\ff(\cdot,\cdot,t)\rangle
\Bigr)\,dt
\\&
=
\langle D_1f_2,\;f_1D_2\ff\rangle
-\langle D_2f_2,\;f_1D_1\ff\rangle
\\&
=\langle
\Adj_{33}Df,\;\ff
\rangle.
\endaligned
$$  
This proves \eqref{equality}.
Now, assume that $\mu:=\Adj_{33} Df\in\M(\Omega)$. 
By Theorem \ref{t:disint},
there exists a disintegration $(\mu_t)_{t\in (0,1)}$ of $\mu$ with respect to $\nu$,
where $\nu$ is as in \eqref{hownu}.
We will first show that $\nu$ is absolutely continuous with respect to the 
Lebesgue measure on $(0,1)$.
 
Assume that $\nu$ is not absolutely continuous.
Then there exists a set $E'\subset (0,1)$ of zero Lebesgue measure 
such that $\nu(E')>0$.
We 
choose a test function
$\psi\in C_0^\infty((0,1)^2)$ such that 
$$
\int_{(0,1)^2} \psi\,d\mu_t >1
$$
for every $t\in E$ where $E$ is a compact subset of $E'$ with $\nu(E)>0.$
This can be done as follows. 
Let $\{\psi_k\}_{k\in \en}$ be a dense sequence in $C^1_0((0,1)^2).$ 
Given any $t$ such that $\mu_t$ is nontrivial measure there is an index $k$ such that
\begin{equation}
\label{positiveslice}
\int_{(0,1)^2} \psi_k\, d\mu_t>1.
\end{equation}
By countable additivity of measures there has to be at least one $k$ such that \eqref{positiveslice} holds for every $t\in E,$ where $E\subset E'$ and $\nu(E)>0.$ Without loss of generality  we may assume that $E$ is compact and, of course, $E$ has 1-dimensional measure zero.

Now, take a sequence $\theta_k$ of smooth functions on $(0,1)$ with compact support
such that 
$0\le \theta_k\le 1$, $\theta_k=1$ on $E$ and $\theta_k\searrow 0$ on
$(0,1)\setminus E$.
Plugging $\theta_k(t)\psi(x_1,x_2)$
into \eqref{disint} and \eqref{integratedadj} we obtain
\eqn{compare}
$$
\aligned
&\int_0^1\Bigl(\int_{(0,1)^2}\theta_k(t)\psi(y)\,d\mu_t(y)\Bigr)\,d\nu(t)
\\&\quad
=\int_0^1\theta_k(t)
\Bigl(\langle D_1f_2(\cdot,t),\;f_1(\cdot,t)D_2\psi\rangle
\\&\qquad-\langle D_2f_2(\cdot,t),\;f_1(\cdot,t)D_1\psi\rangle
\Bigr)\,dt,
\endaligned
$$
where $f(\cdot,t)$ is the function $y\mapsto f(y_1,y_2,t)$.
The integrand on the right is estimated by
$$
C(|D_1f(\cdot,t)|+|D_2f(\cdot,t)|),
$$
which is integrable, see Subsection \ref{ss:slicing}. Since the limit
is zero a.e., the limit on the right hand part of \eqref{compare} is zero
by the Lebesgue dominated convergence theorem.
Similarly we proceed on the left, as 
$$
t\mapsto \Bigl|\int_{(0,1)^2}\psi\,d\mu_t\Bigr|
$$
is integrable with respect to $\nu$, however, here the limit
of integrals is 
$$
\int_{E}\Bigl(\int_{(0,1)^2}\psi\,d\mu_t\Bigr)\,d\nu(t)\ge \nu(E).
$$
This contradiction shows that $\nu$ is absolutely continuous with respect 
to the Lebesgue measure. Let $a$ be the density $d\nu/dt$.
Consider a dense sequence $\{\psi_k\}_{k\in \en}$ in $C^1_0((0,1)^2).$ 
Analogously to \eqref{compare}, for any $k\in\en$  we have
$$
\int_0^1\theta(t)\langle\Det D(f_1(\cdot,t),f_2(\cdot,t)),\psi_k\rangle
=\int_0^1a(t)\theta(t)\langle\nu_t,\psi_k\rangle,
\qquad\theta\in C_0((0,1)).
$$
Hence there exists a Lebesgue null set $N_k\subset(0,1)$
such that 
$$
\langle\Det D(f_1(\cdot,t),f_2(\cdot,t)),\psi_k\rangle=a(t)\langle\nu_t,\psi_k\rangle,
\qquad t\in (0,1)\setminus N_k.
$$
It follows that for each $t\in (0,1)\setminus \bigcup_{k}N_k$ we have
$$
\Det D(f_1(\cdot,t),f_2(\cdot,t))=a(t)\mu_t.
$$
We conclude that the distributions $\Det D(f_1(\cdot,t),f_2(\cdot,t))$ are 
signed Radon measures. 
Since by Remark \ref{r:disint} and \eqref{hownu} 
$$
\int_0^1 a (t)\; d|\mu_t|=\nu((0,1))=|\mu|(\Omega),
$$
the function $t\mapsto \Det D(f_1(\cdot,t),f_2(\cdot,t))$ is integrable. 
\end{proof}



\section{From gradient to degree}\label{s:grad}

Throughout this section we suppose that $u\in BV(\Omega)$
is continuous, in applications this will be the third coordinate of
a BV homeomorphism.

We define
$$
\aligned
h(x)&=(x_1,x_2,u(x)).
\endaligned
$$
Our aim is to prove that 
\eqn{horni}
$$
D_3u(U)=\int_{\er^2}\deg(h,U,z)\,dz
$$
provided that $|h(\partial U)|=0$.

\begin{lemma}\label{l:formulaa}
Let  $U\subset\subset\Omega$ be an open set. 
Let $\eta\in\D(\rn)$ and
$\spt\eta\cap  h(\partial U)=\emptyset$.
Then
$$
\int_{\er^3}\eta(z)\,\deg(h,U,z)\,dz=
\langle D_3u, \eta\circ h\rangle.
$$
\end{lemma}

\begin{proof}
Assume first that $h$ is smooth. Then 
(taking into account that $h(x)=(x_1,x_2,u(x))$) 
the degree formula \eqref{changeofvariables} yields
$$
\aligned
\int_{\rn}\eta(z)\,\deg(h,U,z)\,dz
&=\int_{U}\eta(h(x))\,J_h(x)\,dx
\\
&=\int_{U}\eta(h(x))\,\frac{\partial u(x)}{\partial x_3}\,dx.
\endaligned
$$
Passing to the limit with convolution approximation we obtain 
the required formula.
\end{proof}

\begin{lemma}\label{l:integrablea} Let $U$ be as above and $|h(\partial U)|=0$. Then
the function $\deg(h,U,\cdot)$ is integrable.
\end{lemma}

\begin{proof}
Let $\eta$ be a $C^{\infty}$ function on 
$\er^3$ with $\spt\eta\subset h(\overline U)\setminus h(\partial U)$ and
$|\eta|\le 1$. 
By Lemma \ref{l:formulaa},
$$
\aligned
\Bigl|\int_{\rn}\eta(z)\,\deg(h,U,z)\,dz\Bigr|
&=
\Bigl|\langle D_3u, \eta\circ h\rangle\Bigr|
\\&
\le |Du|(U).
\endaligned
$$
Passing to the supremum over admissible $\eta$ we obtain
$$
\int_{h(\overline U)\setminus h(\partial U)}|\deg(h,U,z)|\,dz\le 
 |Du|(U).
$$
Since $\deg(h,U,\cdot)=0$ on $\er^3\setminus h(\overline U)$
and $|h(\partial U)|=0$, the integrability of $\deg(h,U,\cdot)=0$
is verified.
\end{proof}

\begin{thm}\label{t:D=deg}
Let  $U\subset\subset\Omega$ be an open set and $|h(\partial U)|=0$.
Then
$$
\int_{\er^3}\deg ((x_1,x_2,u),U,z)\,dz= D_3u (U).
$$
\end{thm}
\begin{proof}
Let $\eta_j\in\D(\er^3)$ be smooth functions satisfying
$\spt\eta_j\cap h(\partial U)=\emptyset$ and 
$\eta_j\nearrow 1$ on $h(U)\setminus h(\partial U)$. 
By Lemma \ref{l:formulaa},
$$
\aligned
\int_{\er^3}\deg ((x_1,x_2,u),U,z)\,dz&= 
\lim_{j\to\infty}\int_{\er^3}\eta_j(z)\deg (h,U,z)\,dz
=\lim_{j\to\infty}\langle D_3u,\;\eta_j\circ h\rangle 
\\&=D_3u (U).
\endaligned
$$
The passage to the limit is justified as
$\deg (h,U,\cdot)$ is integrable by Lemma \ref{l:integrablea}
and $D_3u$ is a finite measure.
\end{proof}


\section{From adjugate to degree}\label{s:Adj}

Throughout this section we consider a continuous mapping $f\in \BVd(\Omega,\er^3)$
with a continuous $BV$ inverse.
We define
$$
\aligned
g(x)&=(f_1(x),f_2(x),x_3).
\endaligned
$$
We are going to prove that there is a sufficiently rich 
collection of open sets $U\subset\subset\Omega$
in $\Omega$ such that for each such $U$
we have
\eqn{dolni}
$$
\Adj_{33}Df(U)=\int_{\er^3}\deg(g,U,z)\,dz.
$$

\begin{lemma}\label{l:msz} Let $K\subset \er^3$ be a compact set and $u\ccolon 
K\to\er$ be a continuous function. If $\haus^2(K)<\infty$, then
$\haus^3(\Gamma_u(K))=0$, where $\Gamma_u$ is the mapping
$x\mapsto (x,u(x))$.
\end{lemma}

\begin{proof} Choose $\ep>0$ and find $\delta\in(0,\ep)$ such that 
$$
x,x'\in K,\;|x-x'|<\delta\implies |u(x')-u(x)|<\ep.
$$
Let $(A_j)_j$ be a covering of $K$ by sets of diameter $<\delta$ and 
choose $x_j\in A_j$. A simple partition argument shows that 
$$
\haus_{\infty}^3(\Gamma_u(A_j))\le\haus_{\infty}^3(A_j\times (u(x_j)-\ep,\,u(x_j)+\ep))
\le C\ep(\diam A_j)^2.
$$
Summing over $j$ we obtain
$$
\haus_{\infty}^3(\Gamma_u(K))\le C\ep\sum_j(\diam A_j)^2
$$
and passing to the infimum over all coverings we conclude
$$
\haus_{\infty}^3(\Gamma_u(K))\le C\ep\haus_{\delta}^2(K).
$$
\end{proof}

\begin{lemma}\label{l:formula}
 Let $Q\subset\subset \Omega$ be a cube such that $|g(\partial Q)|=0$.
Let $\eta\in\D(\er^3)$ and
$\spt\eta\cap  g(\partial Q)=\emptyset$.
Let $\Phi\ccolon\er^3\to\er^3$ be a smooth function such that $\Phi_3=0$ and
\eqn{diver}
$$
D_1\Phi_1+D_2\Phi_2=\eta.
$$
Let $\ff\in \D(Q)$ be a test function such that $\ff=1$ on
$\{\eta\ne 0\}$.
Then
$$
\int_{\er^3}\eta(z)\,\deg(g,Q,z)\,dz=
\langle Df_2,\ (\Phi_1\circ g)\;  \nabla\ff\times\e_3\rangle
-\langle Df_1,\ (\Phi_2\circ g)\; \nabla\ff\times\e_3\rangle.
$$
\end{lemma}

\begin{proof}
In this proof we don't need invertibility. Thus we may use an  approximation argument and assume first  that $g$ is smooth.
Then a direct computation together with interchangeability of second derivatives gives 
$$
\dive\bigl(
\Phi_1\circ g\;\nabla g_2\times\nabla g_3
+\Phi_2\circ g\;\nabla g_3\times\nabla g_1
+\Phi_3\circ g\;\nabla g_1\times\nabla g_2\bigr)=((\dive\Phi)\circ g)\;J_g
$$
so that (taking into account that $\dive\Phi=\eta$,
$\Phi_3=0$ and $\nabla g_3=\e_3$), the degree formula \eqref{changeofvariables} yields
$$
\aligned
\int_{\er^3}\eta(z)\,\deg(g,Q,z)\,dz
&=\int_{Q}\eta(g(x))\,J_g(x)\,dx
=\int_{Q}\eta(g(x))\,\ff(x)\,J_g(x)\,dx
\\
&=\int_Q\ff(x)\dive\bigl(\Phi_1\circ g\;\nabla g_2\times\nabla g_3
+\Phi_2\circ g\;\nabla g_3\times\nabla g_1\bigr)\; dx
\\
&=\int_{Q}\Phi_1\circ g\;\nabla g_2\cdot  \nabla\ff\times\e_3\,dx
-\int_{Q}\Phi_2\circ g\;\nabla g_1\cdot \nabla\ff\times\e_3\,dx.
\endaligned
$$
Passing to the limit with convolution approximations of true $g$ we obtain 
the required formula.
\end{proof}

\begin{lemma}\label{l:integrable}  Let $Q\subset\subset \Omega$ be a cube such that $|g(\partial Q)|=0$.
 Then
the function $\deg(g,Q,\cdot)$ is integrable.
\end{lemma}

\begin{proof}
It is easy to see that mappings
$$
g(x)=(f_1(x),f_2(x),x_3)\text{ and }h(y)=(y_1,y_2,(f^{-1})_3(y))
$$
satisfy $g=h\circ f$ and that the degree of a homeomorphism $f$ is $1$. 
By the degree composition formula Lemma \ref{composition}, we thus have 
$\deg(g,Q,\cdot)=\deg(h,f(Q),\cdot)$. 
Now, the conclusion follows from Lemma \ref{l:integrablea}.
\end{proof}

\begin{definition}\label{d:good} 
Let $\bar x_i\in\er$, $i\in\{1,2,3\}$.
We say that $H=H_{i,\bar x_i}:=\{x\ccolon x_i=\bar x_i\}$ is a \textit{good plane}
if the following properties hold:
\eqn{good1}
$$
|Df|(H\cap \Omega)=0,\ |\Adj Df|(H\cap \Omega)=0.
$$
\eqn{good2}
$$
|g(H\cap \Omega)|=0.
$$
\eqn{good3}
$$
\limsup_{r\to 0} \frac{|Df|\bigr(\Omega\cap \{x\ccolon |x_i-\bar x_i|<r\}\bigl)
}{r}<\infty.
$$
\eqn{good4}
$$
\aligned
\text{If }i=1, &\\
\limsup_{r\to 0} &\frac{|Df(\cdot,\cdot,x_3)|
 \bigl(\Omega\cap((\bar x_1-r,\bar x_1+r)\times \er\times \{x_3\})\bigr)
 }{r}<\infty
\\&
\text{ for a.e. }x_3\in \er.
\endaligned
$$
\eqn{good5}
$$
\aligned
\text{If }i=2, &\\
\limsup_{r\to 0} &\frac{|Df(\cdot,\cdot,x_3)|
 \bigl(\Omega\cap(\er\times(\bar x_2-r,\bar x_2+r)\times \{x_3\})\bigr)
 }{r}<\infty
\\&
\text{ for a.e. }x_3\in \er.
\endaligned
$$
We say that a cube $Q(\bar x,r)\subset\er^3$ is a \textit{good cube}
if all its faces are subsets of good planes.

It is obvious that almost all $\bar x_i$ satisfy
\eqref{good1} and \eqref{good3}. The validity of \eqref{good2} and \eqref{good4} -- \eqref{good5}
for almost all $\bar x_i$ will be verified in Lemma 
\ref{l:measzero} and Lemma \ref{l:fubini}.

Now, consider $\bar z\in\er^3$ such that for each $i=1,2,3$ and each dyadic rational
$q$, the plane $\{x\ccolon x_i=\bar z_i+q\}$ is good. We see that almost each $\bar z\in\er^3$ has this property. It follows that we can consider arbitrarily
fine regular translated-dyadic partitions of $\er^3$ consisting of good cubes.
For simplicity (and without loss of generality), we assume that
the origin of coordinates has the property described above and thus all
dyadic cubes $\{(2^{-k}z_1,2^{-k}(z_1+1))\times(2^{-k}z_2,2^{-k}(z_2+1))\times (2^{-k}z_3,2^{-k}(z_3+1))\}$,
$z\in \zet^3$, are 
good.
\end{definition}

\begin{lemma}\label{l:measzero}
Almost every $\bar x_i\in\er$ satisfies \eqref{good2}.
\end{lemma}

\begin{proof} By \cite[Theorem 3.1]{HKL}, 
for almost every $\bar x_i\in\er$ we have $\haus^2 (f(H))<\infty$, where
$H=\{x\ccolon x_i=\bar x_i\}$. Pick such $\bar x_i$. 
Let $K\subset H$ be a compact 
set.
Then $\haus^2 (f(K))<\infty$ and by Lemma \ref{l:msz} for $u=(f^{-1})_3$ we have 
$$
\haus^3\bigl(\{(f_1(x),f_2(x),f_3(x),x_3):\ x\in K\}\bigr)=0.
$$
Hence 
$$
0=\haus^3\bigl(\{(f_1(x),f_2(x),x_3):\ x\in K\}\bigr) =|g(K)|.
$$
\end{proof}

\begin{lemma}\label{l:fubini}
Almost every $\bar x_1\in \er$ satisfies
\eqref{good4} and almost every  $\bar x_2\in \er$ satisfies
\eqref{good5}.
\end{lemma}

\begin{proof}
It is enough to consider \eqref{good4}.
We may assume that $\Omega$ is the cube $(0,1)^3$. 
We consider the function 
$$
\psi(x_1,x_3)=|Df(\cdot,\cdot,x_3)|((0,x_1)\times (0,1)\times\{x_3\}).
$$
It can be rewritten as
$$
\psi(\bar x_1,x_3)=
\sup_{j\in\en}
\iint_{(0,\bar x_1)\times(0,1)}f(x_1,x_2,x_3)\dive \ff_j(x_1,x_2)\,dx_1\,dx_2,
$$
where $\{\ff_j\}$ is a dense sequence in the collection of 
all $\ff\in\D((0,1)^2,\er^2)$ with $\sup_{(0,1)^2}|\ff|\le 1$.
Therefore $\psi$ is measurable. Since $\psi$ is 
is increasing in $x_1$, we can express the upper partial derivative
of $\psi$ at $(x_1,x_3)$ with respect to $x_1$ as
$$
\overline D_1\psi(x_1,x_3)=\inf_{m\in\en}\sup_{q\in\qe\cap (-\frac1m,\frac1m)\setminus \{0\}}
\frac{\psi(x_1+q,x_3)-\psi(x_1,x_3)}{q},
$$
where $\qe$ is the set of all rationals, similarly for the lower partial derivative.
It follows that the set where the partial derivative of $\psi$ at $(x_1,x_3)$ with respect to $x_1$
exists is measurable.
Taking into account again that $\psi$ is increasing in $x_1$, we infer that
there exists a set $N\subset \er^2$
of measure zero
such that
the  partial derivative $\frac{\partial\psi}{\partial x_1}$
exists outside $N$. Now, \eqref{good4} is satisfied at $\bar x_1$ if 
the one-dimensional measure of $N\cap (\{\bar x_1\}\times \er)$ is 
zero, which is true for a.e.\ $\bar x_1$ by the Fubini theorem.
\end{proof}

\subsection{Construction}\
Let  $\bar x\in \Omega$ and $0<r<r_0=\dist(\bar x,\partial\Omega)$. 
Let $Q=Q(\bar x,r)$ be a good cube.
Let $\eta_0$ be a smooth function with compact support such that
$0\le \eta_0\le 1$ and $\eta_0=1$ on $g(\Omega)$. 
As in the proof of Theorem \ref{t:twodim},
consider a 
sequence $\eta_k$ of smooth functions such that $\eta_k=0$ on a neighborhood of
$g(\partial Q)$, $k=1,2,\dots$,
$0\le \eta_1\le \eta_2\le \dots\le \eta_0$ and 
$\eta_k\to\eta_0$ a.e. 
Let $K$ be as in \eqref{howK}. 
Set
$$
 \Phi^{(k)}=K*\eta_k,\qquad k=0,1,2,\dots.
$$
Then 
$$
D_1\Phi_1^{(k)}+D_2\Phi_2^{(k)}=\eta_k \quad \text{ on }\overline{g(\Omega)}.
$$
Further,
for almost every $x_3\in\er$, $\eta_k\to\eta_0$ in 
$L^3(\er^2\times\{x_3\})$. 
From Lemma \ref{l:Lp} we obtain that
$\Phi^{(k)}(z)\to \Phi$ 
uniformly
on $\overline {g(\Omega)}\cap (\er^2\times\{x_3\})$. 
Next, we find $\rho_k\nearrow r$ such that 
$$
\eta_k=0 \text{ on }g(Q\setminus Q(\bar x,\rho_k))
$$
and test functions $\ff_k\in\D(Q)$ such that $0\le \ff_k\le 1$,
$\ff_k=1$ on $ Q(\bar x,\rho_k)$ and 
$$
|\nabla\ff_k|\le \frac{C}{r-\rho_k}.
$$

\begin{lemma}\label{l:tezsi}
Let  $\bar x\in \Omega$ and $0<r<r_0=\dist(\bar x,\partial\Omega)$. 
Let $Q=Q(\bar x,r)$ be a good cube.
Then
$$
\lim_{k\to\infty}\Bigl(
\langle Df_2,\ (\Phi_1^{(k)}\circ g) \; \nabla\ff_k\times\e_3\rangle
-\langle Df_1,\ (\Phi_2^{(k)}\circ g) \; \nabla\ff_k\times\e_3\rangle
\Bigr)
=\ADJ_{33} Df (Q).
$$
\end{lemma}

\begin{proof}
We interpret the symbols like $\adj_{ij}Df(\cdot,\cdot,x_3)$ so that the differential operator 
is applied to the function of two variables $y\mapsto f(y_1,y_2,x_3)$.
By Lemma \ref{l:atx1}, Fubini theorem, \eqref{good4} and \eqref{good5} (recall that $Q$ is a good cube) 
\eqn{longcomp}
$$
\aligned
&\langle Df_2,\ (\Phi_1^{(k)}\circ g) \; \nabla\ff_k\times\e_3\rangle
-\langle Df_1,\ (\Phi_2^{(k)}\circ g) \; \nabla\ff_k\times\e_3\rangle
\\&\quad
=
\int_{\bar x_3-r}^{\bar x_3+r}
\Bigl(
\langle Df_2(\cdot,\cdot,x_3),\ (\Phi_1^{(k)}\circ g) \; \nabla\ff_k(\cdot,\cdot,x_3)\times\e_3\rangle
\\&\qquad-\langle Df_1(\cdot,\cdot,x_3),\ (\Phi_2^{(k)}\circ g) \; \nabla\ff_k(\cdot,\cdot,x_3)\times\e_3\rangle
\Bigr)\,dx_3
\\&\quad
=
-\int_{\bar x_3-r}^{\bar x_3+r}
\sum_{i,j=1}^2\big\langle \adj_{ij}Df(\cdot,\cdot,x_3)),(\Phi_j^{(k)}\circ g) \; D_i(\ff_k(\cdot,\cdot,x_3))\big\rangle
\,dx_3
\\&\quad
\\&\quad
\to
\int_{\bar x_3-r}^{\bar x_3+r}\Det(Df_1(\cdot,\cdot,x_3),Df_2(\cdot,\cdot,x_3))(Q)\,dx_3
=\ADJ_{33} Df (Q).
\\&\quad
\endaligned
$$
To justify the passage to limit 
in \eqref{longcomp}
under the integral sign we need the 
pointwise convergence a.e., which is verified by Lemma \ref{l:atx1},
and a convergent majorant.
By Corollary \ref{c:Nest},
for almost each $x_3\in (\bar x_3-r,\,\bar x_3+r)$ we have
$$
\aligned
&
\Bigl|
\langle Df_2(\cdot,\cdot,x_3),\ (\Phi_1^{(k)}\circ g) \; \nabla\ff_k(\cdot,\cdot,x_3)\times\e_3\rangle
\\&\qquad
-\langle Df_1(\cdot,\cdot,x_3),\ (\Phi_2^{(k)}\circ g) \; \nabla\ff_k(\cdot,\cdot,x_3)\times\e_3
\rangle
\Bigr|
\\
&\quad\le C\int_{\er^2} \bigl|\deg((f_1(\cdot,\cdot,x_3),f_2(\cdot,\cdot,x_3)),
Q_2((\bar x_1,\bar x_2),r),(y_1,y_2))\bigr|\,dy_1\,dy_2\\
&\quad\le C\int_{\er^2} N((f_1(\cdot,\cdot,x_3),f_2(\cdot,\cdot,x_3)),
Q_2((\bar x_1,\bar x_2),r),(y_1,y_2))\,dy_1\,dy_2\\
&\quad\le C\haus^2 \bigl(f(Q_2((\bar x_1,\bar x_2),r)\times \{x_3\})\bigr),
\endaligned
$$
where the estimate of degree by multiplicity is from \cite[Lemma 6.1]{HKL}. 
For the last inequality see \cite[Theorem 7.7]{Mattila}.
From
\cite[Theorem 3.1]{HKL} we deduce that the function
$$
x_3\mapsto \haus^2 \bigl(f(Q_2((\bar x_1,\bar x_2),r)\times \{x_3\})\bigr)
$$
is integrable over $(\bar x_3-r,\bar x_3+r)$.
\end{proof}

\begin{thm}\label{t:Adj=deg}
Let $Q\subset\subset\Omega$ be a good cube.
Then
$$
\int_{\er^3}\deg (g,Q,z)\,dz=\Adj_{33} Df (Q).
$$
\end{thm}
\begin{proof}
By Lemma \ref{l:tezsi} and Lemma \ref{l:formula},
$$
\aligned
\Adj_{33} Df (Q)&=
\lim_{k\to\infty}
\Bigl(
\langle Df_2,\ (\Phi_1^{(k)}\circ g) \; \nabla\ff_k\times\e_3\rangle
-\langle Df_1,\ (\Phi_2^{(k)}\circ g) \; \nabla\ff_k\times\e_3\rangle
\Bigr)
\\&=
\lim_{k\to\infty}\int_{\er^3}\eta_k(z)\deg (g,Q,z)\,dz
=\int_{\er^3}\deg ((f_1,f_2,x_3),Q,z)\,dz.
\endaligned
$$
The passage to the limit in the last equality is justified as
$\deg (g,Q,\cdot)$ is integrable by Lemma \ref{l:integrable}.
The equality $\Adj_{33} Df (Q)=\ADJ_{33} Df (Q)$ follows from
Proposition \ref{aapo}.
\end{proof}


\section{The proof of the main results}\label{s:cor}

\subsection{Formula \eqref{main} -- conclusion}

Now we are ready to prove Theorem \ref{t:main}.

\begin{definition}
We say that a set $F$ is a \textit{closed dyadic figure} if
it is a finite union of closed dyadic cubes. An interior of a closed 
dyadic figure is called an \textit{open dyadic figure}.
\end{definition}

\begin{proof}[Proof of Theorem \ref{t:main}] 
Recall that for
symmetry reasons we 
demonstrate 
the proof for $i=j=3$. 
 Let $Q\subset\subset \Omega$ be a good cube.
We apply Theorem \ref{t:Adj=deg} to $f$ and
Theorem \ref{t:D=deg} to the mapping $y\mapsto (y_1,y_2,(f^{-1})_3(y))$. 
 We obtain
\eqn{measure}
$$
D_2(f^{-1})_3(f(U))
=\int_{\er^3}\deg(g\circ f^{-1},f(U),z)\,dz
=\int_{\er^3}\deg(g,U,z)\,dz
=\Adj_{33} Df (U)
$$
for $U=Q$ taking into account the degree composition formula Lemma \ref{composition}, which justifies 
the second equality in \eqref{measure}. 
If $U$ is an open dyadic figure, then $|g(\partial U)|=0$ holds as well and we can use 
the additivity of the degree.
A general open set can be written as the union of an increasing sequence of 
open dyadic figures.

\end{proof}

\subsection{Particular cases}

In this subsection we prove Theorem \ref{cor}.

\begin{proof}[Proof of Theorem \ref{cor}] 
{\sc Step 1:}
We first show that $\Adj_{i,j}Df\in \M(\Omega)$ 
for each $i,j\in \{1,2,3\}$. We  demonstrate this only for $i=j=3$ as other cases are identical.

Assume that (a) holds.
We claim that 
\eqn{pq}
$$
\langle\Adj_{33}Df,\ff\rangle=\int \varphi \det \bigl(D_jf_i\bigr)_{i,j=1,2}\,dx,
\qquad \ff\in \D(\Omega).
$$
This fact demonstrates that $\Adj_{33}Df$ is a measure, as 
$\det \bigl(D_jf_i\bigr)_{i,j=1,2}$ is an measure (it is even an $L^1$ function).
To prove \eqref{pq},
we first assume that $f_1$ and $f_2$ are smooth, then it is just integration by parts.
Next step is to assume that $f_2$ is smooth. By mollification we obtain
a sequence $\{f^{(k)}_2\}_k$ such that $Df^{(k)}_2$ converge weak* to $Df_2$,
so it is easy to observe that \eqref{pq} holds in this case as well.
 
Finally, we use the preceding step and
mollify $f_1$ to obtain a sequence $\{f^{(k)}_1\}_k$ such that 
$Df^{(k)}_1$ converge to $Df_1$, strongly if $p_1<\infty$ or weak* if $p_1=\infty$.
In any case we can conclude that \eqref{pq} holds for the limit function.

Now assume (b) holds. The statement is trivial if $f_1$ and $f_2$ are smooth.
Assume e.g.\ that $f_2$ and $f_3$ are smooth.
We proceed as in the first two steps of (a) replacing \eqref{pq} by
$$
\langle\Adj_{33}Df,\ff\rangle=\langle Df_1, \ff Df_2\times \e_3\rangle.
$$

{\sc Step 2:}
By Theorem \ref{regularity}, 
to prove that $f^{-1}$ has 
bounded variation we need to show that $f$ satisfies the finite Lebesgue area condition of Section \ref{sec:LebArea}. That is,  $f\circ \kappa_i^t$ has finite Lebesgue area for almost every $t$ and every $i=1,2,3.$  Recall that functions $\kappa_i^t$ were defined in the beginning of Section \ref{sec:defs}.
Notice that for almost every $t$ we have $f\circ \kappa_i^t\in \BVd.$ We pick such $t$
and denote $g=(g_1,g_2,g_3)=f\circ \kappa_i^t$. Without loss of generality we may assume that
$g_1\in \BVd$ and $g_2,\,g_3\in C^1(\bar U)$, where $U=(\kappa_i^t)^{-1}(\Omega)\subset\er^2$. 
The finiteness of Lebesgue area under assumption (a) is due to Morrey (see \cite[Section 5.13]{C}). 
Next we show the finiteness under assumption (b). The proof is a simplification of Morrey's proof and the reasoning could easily be modified to the case of (a).



Let $g^k$ be a sequence such that 
$g^k\rightarrow g$ and $Dg^k_i\rightarrow Dg_i,$ $i=2,3$  uniformly and
$|Dg^k_1|(U)\rightarrow|Dg_1|(U)$ (see \cite[Theorem 3.9]{AFP}). 
Then
$$
\int_U|\det (Dg^k_{1},Dg^k_{2})|\leq C \|Dg^k_1\|_1\;\|Dg^k_2\|_{\infty}\le C,
$$
similarly for other choices of coordinates.
Since by
\cite[Sections 5.10 and Section 5.13 Note 2]{C}
$$
L(g)\leq \liminf_{k} L(g^k),
$$
and
$$
L(g^k)\leq\int_U \Bigl(\sum_{1\le i<j\le 3}|\det (Dg^k_{i},Dg^k_{j})|^2\Bigr)^{1/2}\,dy
\le \sum_{1\le i<j\le 3}\int_U |\det (Dg^k_{i},Dg^k_{j})|\,dy,
$$
we have verified that $L(g)<\infty$. 
\end{proof}


\subsection{Absolutely continuous part of $\Adj Df(x)$}
\newdimen\vintkern
\vintkern12pt
\def\vint{{-}\kern-\vintkern\int}

Let $B$ be the unit disc in $\er^2$ and 
$u,v$ are continuous $BV$ functions on $B$. We express $u$ and 
$v$ in polar coordinates, writing 
$$
\bar u(\rho,t)=u(\rho \cos t,\,\rho \sin t),\quad
\bar v(\rho,t)=v(\rho \cos t,\,\rho \sin t).
$$
Then
$$
\int_{\partial B(0,\rho)}u\,dv
$$
is the Riemann-Stieltjes integral 
$$
\int_0^{2\pi}\bar u(\rho,\cdot)\,d\bar v(\rho,\cdot).
$$

\begin{lemma}\label{l:green} Let $h=(u,v)\ccolon B(0,1)\to\er^2$
be a continuous BV mapping. Suppose that $\djac_h\in\M(B(0,1))$.
Then for a.e.\ $\rho\in (0,1)$ we 
have
\eqn{green}
$$
\djac _h(B(0,\rho))=\int_{\partial B(0,\rho)}u\,dv.
$$
\end{lemma}

\begin{proof} Let $\eta$ be a smooth function on 
$[0,1]$ such that $\eta'(0+)=\eta'(1-)=0$, $\eta(1)=0$, $\eta(0)=1$ and 
$\eta'<0$ on $(0,1)$. Let $\ff(x)=\eta(|x|)$ and $\mu=\djac_h$.
We have
$$
\aligned
\int_0^1|\eta'(r)|\,\mu(B(0,r))\,dr&=
\int_0^1\mu(\{\ff>t\})\,dt
\\&=
\int_{B(0,1)}\Bigl(\int_0^{\eta(|x|)}\,dt\Bigr)\,d\mu(x)
=\int_{B(0,1)}\ff\,d\mu=
\\&
=-\langle Dv, *u\nabla\ff\rangle
\\&
=\int_0^1|\eta'(r)|\int_{\partial B(\rho)}u\,dv,
\endaligned
$$
where the last equality is obtained by slicing in the polar coordinates, see \cite[Theorem 3.107]{AFP}.
Varying $\eta$ we obtain \eqref{green}.
\end{proof}

\begin{lemma}\label{l:bvconv}
Let $u_j$, $v_j$ be continuous functions on $\partial B(0,\rho)$,
$j=1,2,\dots,\infty$.
Let $u_j$, $v_j$ converge to $u_{\infty},v_{\infty}$ strongly
in $BV(\partial B(0,\rho))$. Then
$$
\int_{\partial B(0,\rho)}u_j\,dv_j\to\int_{\partial B(0,\rho)} u_\infty\,dv_\infty.
$$
\end{lemma}

\begin{proof} It is an immediate consequence of the fact that
strong convergence of continuous functions in the $BV$ norm implies
the uniform convergence (if the dimension is one).
\end{proof}

\begin{thm}\label{t:absjac}
Let $U\subset\er^2$ be an open set and $h\in BV(U)$ be continuous.
Suppose that $\djac_h\in\M(U)$.
Then the absolutely continuous part of $\djac_h$ is $J_h$ for 
a.e. $x\in \Omega$.
\end{thm}

\begin{proof}
Write $h$ is coordinates as $h=(u,v)$.
Let $\mu=\djac_h$ and $\theta$ be the density of the absolutely continuous
part of $\mu$. Recall that the approximative derivative $\nabla h$ is the density of the 
absolutely continuous part of $Dh$ and $J_h=\det\nabla h$.
Further, $D_sh$ is the singular part of $Dh$ and $\mu_s$ is the singular
part of $\mu$.

Let $x_0$ be a point satisfying the following properties:
\eqn{x01}
$$
x_0\text{ is a Lebesgue point for }
\ \nabla h \text{ and }\theta,
$$
\eqn{x02}
$$
\lim_{r\to 0}\frac{|\mu_s|(B(x_0,r))+|D_sh|(B(x_0,r))}{|B(x_0,r)|}=0
$$
Then almost every point $x_0\in U$ has the desired properties, (see \cite[Theorems 2.12 and 2.17]{Mattila}).
For simplicity assume that $x_0=h(x_0)=0$.
Choose a sequence $r_j\searrow 0$ such that $B(0,r_1)\subset\subset\Omega$
and denote
$$
\aligned
h_j(y)&=(u_j(y),v_j(y))=\frac{1}{r_j}h(r_jy),\qquad y\in B(0,1),\\
h_{\infty}(y)&=(u_{\infty}(y),v_{\infty}(y))=\nabla h(0)y.
\endaligned
$$
Now, consider a radius $\rho\in (0,1)$ with the following properties:
\eqn{rho1}
$$
\int_{\partial B(0,\rho r_j)}u\,dv = \mu(B(0,\rho r_j)),
$$
\eqn{rho2}
$$
u_j\to u_{\infty}\text{ and }v_j\to v_{\infty}
\text{ strongly in }BV(\partial B(0,\rho)),
$$
The proof of existence of such a radius is postponed for a while.
We have by \eqref{x02} and Lemma \ref{l:bvconv}
$$
\aligned
\theta(0)&=\lim_{j\to\infty}\vint_{B(0,\rho r_j)}\theta(y)\,dy
=\lim_{j\to\infty}\frac{\mu(B(0,\rho r_j))}{|B(0,\rho r_j)|}
\\&
=\lim_{j\to\infty}\frac1{|B(0,\rho r_j)|}\int_{\partial B(0,\rho r_j)}u\,dv
=\lim_{j\to\infty}\frac1{|B(0,\rho)|}\int_{\partial B(0,\rho)}u_j\,dv_j
\\&
=\frac1{|B(0,\rho)|}\int_{\partial B(0,\rho)}u_\infty\,dv_\infty
=\vint_{B(0,\rho)} J_{h_{\infty}}(y)\,dy
\\&=J_h(0).
\endaligned
$$
Now, by Lemma \ref{l:green}, almost every $\rho\in (0,1)$ satisfies
\eqref{rho1}.
 To show \eqref{rho2} we show first that $h_j$ converges to $h_\infty$ strongly in $BV(B).$
As the $L^1$-convergence follows from the definition of approximate differentiability it suffices to consider the convergence of the derivative.

By \cite[Remark 3.18]{AFP} we have for every Borel set $A\subset B$ 
\eqn{scaleDh}
$$
Dh_j(A)=\frac{1}{r_j^2}\left(\int_{r_jA}\nabla h(x)dx + D_sh(r_jA)\right).
$$
We now establish the strong convergence of the derivative. Using \eqref{scaleDh} we estimate
\eqn{strongB}
$$
\aligned
|Dh_j-Dh_\infty|(B)&= \sup_{\{\varphi\in C_0(B)\ccolon |\varphi|\leq1\}} \int_B \varphi d(Dh_j-Dh_\infty)\\&\leq 
\sup_\varphi \frac{1}{r_j} \int_{r_jB}|\nabla h(x)-\nabla h(0)|dx + \frac{|D_sh|(r_jB)}{r_j^2}\rightarrow 0 
\endaligned
$$
Here the convergence on the last step follows from \eqref{x01} and \eqref{x02}.
Finally the strong convergence on almost every $\rho\in(0,1)$ follows from this and \cite[Theorem 3.103]{AFP} applied to polar coordinates. 

\end{proof}

\begin{proof}[Proof of Theorem \ref{acpart}]
This follows from Theorem \ref{t:absjac} using
Lemma \ref{l:absdisint}.
\end{proof}

\end{document}